\documentclass[reqno,12pt]{amsart}
\usepackage{a4wide,color,eucal,enumerate,mathrsfs}
\usepackage[normalem]{ulem}
\usepackage{amsmath,amssymb,epsfig,amsthm} 
\usepackage[latin1]{inputenc}
\usepackage{psfrag}
\usepackage{hyperref,cleveref,tikz}

\def\bint{{\ifinner\rlap{\bf\kern.35em--}
\int\else\rlap{\bf\kern.45em--}\int\fi}\ignorespaces}

\def\bbint{{\ifinner\rlap{\bf\kern.35em--}
\hspace{0.078cm}\int\else\rlap{\bf\kern.45em--}\int\fi}\ignorespaces}

\newcommand{\R}{\mathbb R}

\newcommand{\eps}{\varepsilon}

\newcommand{\crit}{\operatorname{Crit}}

\newtheorem{thm}{Theorem}[section]
\newtheorem{lem}[thm]{Lemma}
\newtheorem{prop}[thm]{Proposition}
\newtheorem{cor}[thm]{Corollary}
\newtheorem{defn}[thm]{Definition}

\numberwithin{equation}{section}

\theoremstyle{remark}
\newtheorem{rem}[thm]{Remark}

\def\bint{{\ifinner\rlap{\bf\kern.35em--}
\int\else\rlap{\bf\kern.45em--}\int\fi}\ignorespaces}

\usepackage{graphicx}
\usepackage{import}
\usepackage{xifthen}
\usepackage{pdfpages}
\usepackage{transparent}

\title[An Almgren-type formula for planar $p$-harmonic functions]{An Almgren-type formula for planar $p$-harmonic functions}
\author{Yi Ru-Ya Zhang}
\dedicatory{Dedicated to Nicola Fusco on the occasion of his 70th birthday}
\date{\today}

\address{State Key Laboratory of Mathematical Sciences, Academy of Mathematics and Systems Science, Chinese Academy of Sciences, Beijing 100190, China}
\address{Academy of Mathematics and Systems Science, the Chinese Academy of Sciences, Beijing 100190, China}
\email{yzhang@amss.ac.cn}
 \thanks{ The author is funded by the National Key R\&D Program of China (Grant No. 2025YFA1018400 \&  No. 2021YFA1003100), NSFC Grant No. 12288201 \& No. 12571128, the Chinese Academy of Sciences, and CAS Project for Young Scientists in Basic Research, Grant No. YSBR-031. 
}


\begin{document}

\begin{abstract}
For a nonconstant planar $p$-harmonic function $u$, with $1 < p < \infty$, and according to previous results, most notably Aronsson's fundamental analysis of the hodograph representation in a neighborhood of an isolated critical point, we introduce the flux-normalized frequency
$$
  N_*(x_0,r)  =  \frac{r\displaystyle\int_{B_r(x_0)} |D u|^p\,dx}  {\displaystyle\int_{\partial B_r(x_0)} |D u|^{p-2}(u-u(x_0))^2\,dS}.
$$
This ratio constitutes the natural counterpart of Almgren's frequency function in the nonlinear setting, as it recovers precisely the degree of homogeneity for every homogeneous $p$-harmonic profile. We further establish the corresponding gauge-corrected version of Almgren-type monotonicity and reinterpret, in terms of the frequency function, the classical planar unique continuation property. In addition, we derive a quantitative pinching estimate for the flux-normalized frequency and identify the specific obstruction that prevents a straightforward extension of the method to higher dimensions.
\end{abstract}


\maketitle

\section{Introduction}

Almgren's frequency function \cite{A1979} is one of the most useful ways of measuring homogeneity. If $u$ is harmonic in a ball $B_R(x_0)\subset \mathbb R^n$
and $u(x_0)=0$, then
$$
 N_2(x_0,r)  = \frac{r\int_{B_r(x_0)}|D u|^2\,dx} {\int_{\partial B_r(x_0)}u^2\,dS}
$$
is nondecreasing in $r$, where $dS$ denotes the surface measure on the sphere. Moreover $N_2$ is constant exactly on homogeneous harmonic functions, and the limiting value $N_2(x_0,0+)$ is the vanishing order of $u$ at $x_0$. This observation, going back to Almgren, has become a central tool in unique continuation and in the analysis of nodal and singular sets; see also the work of Garofalo--Lin \cite{GL1986} for the elliptic unique-continuation setting.

The purpose of the present note is to formulate the corresponding
frequency calculation for planar $p$-harmonic functions,
$$
\operatorname{div}(|D u|^{p-2}D u)=0, \qquad 1<p<\infty .
$$
The nonlinear equation has an immediate difficulty which is absent when
$p=2$: The natural elliptic weight
$$
 A(x):=|D u(x)|^{p-2}
$$
depends on the unknown solution and degenerates or becomes singular at
critical points. Thus the usual harmonic height
$\int_{\partial B_r}u^2\,dS$ is no longer the correct object if one
wants an identity calibrated to the homogeneity of $p$-harmonic
profiles.

There are two related, important, but distinct nonlinear precedents. First, stationary points of the $p$-energy satisfy the standard
energy monotonicity formula
$$
  \frac{d}{dr}\left(  r^{p-n}\int_{B_r(x_0)} |Du|^p\,dx
  \right) = p r^{p-n}\int_{\partial B_r(x_0)} |Du|^{p-2}(\partial_\nu u)^2\,dS \ge 0 ,
$$
which follows from the Rellich--Pohozaev identity; see, for instance,
Hardt--Lin \cite[Lemma 4.1]{HL1987} and the exposition in
\cite[Section 4]{M2009}. This formula measures the scaling of the
$p$-energy, but it is not an Almgren frequency quotient: It does not contain
a boundary height and does not return the homogeneity degree of a homogeneous
$p$-harmonic profile. Second, Granlund--Marola
\cite{GM2014} studied the unique-continuation problem for the
$p$-Laplace equation through nonlinear analogues of Almgren's frequency,
such as quotients with boundary height 
$$\int_{\partial B_r}|u|^p\,dS,$$
and proved conditional consequences from boundedness of such frequencies.

In the current manuscript we use the following flux-normalized height. After subtracting the value
at the center, write
$$
  v(x)=u(x)-u(x_0),
$$
and define
$$
  D(x_0,r)=\int_{B_r(x_0)}|D u|^p\,dx,
  \qquad H(x_0,r)=\int_{\partial B_r(x_0)} |D u|^{p-2}v^2\,dS .
$$
Whenever $H(x_0,r)>0$, set
$$
 N_*(x_0,r) =  \frac{rD(x_0,r)}{H(x_0,r)} .
$$
This normalization is forced by homogeneity. Indeed, if
$$
  u(x_0+r\theta)-u(x_0)=r^\lambda \Phi(\theta),
$$
then
$$
 \partial_\nu u=\frac{\lambda}{r}\,v ,
$$
and testing the $p$-Laplace equation by $v$ gives
$$
\int_{B_r(x_0)}|D u|^p\,dx
  = \int_{\partial B_r(x_0)}
|D u|^{p-2}v\,\partial_\nu u\,dS
 =\frac{\lambda}{r} \int_{\partial B_r(x_0)}|D u|^{p-2}v^2\,dS .
$$
Consequently $N_*(x_0,r)\equiv\lambda$ on every homogeneous
$p$-harmonic profile. This is the main reason for us to consider the weighted boundary height in the definition of $H$.

The price of this exact normalization is that the derivative of $H$
contains a new term. Namely,
$$
\frac{d}{dr}H(x_0,r)= \frac{1}{r}H(x_0,r)+2D(x_0,r)+K(x_0,r),
$$
where
$$
 K(x_0,r) = \int_{\partial B_r(x_0)} v^2\,\partial_\nu(|D u|^{p-2})\,dS .
$$
The term $K$ vanishes in the harmonic case $p=2$, but for
$p\ne2$ it is the only obstruction to the classical square formula.
Indeed, the elementary computation Proposition~\ref{thm:exact-identity} below gives the exact identity
$$
  N_*'(x_0,r)+N_*(x_0,r)\Theta(x_0,r)
 = \frac{pr}{H(x_0,r)} \int_{\partial B_r(x_0)}
|D u|^{p-2} \left( \partial_\nu u-\frac{N_*(x_0,r)}{r}v
  \right)^2\,dS ,
$$
where
$$
 \Theta(x_0,r)   =\frac{K(x_0,r)}{H(x_0,r)} - \frac{(p-2)(N_*(x_0,r)-1)}{r}.
$$
Thus the gauge-corrected quantity
$$
 \exp\left(\int_{r_0}^r \Theta(x_0,s)\,ds\right)N_*(x_0,r)
$$
is monotone whenever $\Theta(x_0,\cdot)\in L^1(0,r_0)$; see Corollary~\ref{cor:gauge-formal}.

The point at which the proof becomes genuinely two-dimensional is the
integrability of this obstruction at critical points. In the plane this follows from the fundamental representation theorem of Aronsson
\cite{A1989}, which describes a $p$-harmonic function near a
zero of its gradient by means of the stream function and the hodograph
transform. In particular, if $x_0$ is an isolated critical point, then
there is a nonzero homogeneous $p$-harmonic function $P$ and numbers
$\lambda>1$, $\delta>0$ such that, after subtracting $u(x_0)$,
$$
  u(x_0+r\theta)=P(r\theta)+O(r^{\lambda+\delta}),
 \qquad  D u(x_0+r\theta)=D P(r\theta)  +O(r^{\lambda-1+\delta}),
$$
uniformly in $\theta\in S^1$. The same representation gives the following important logarithmic-gradient asymptotic
$$
 \partial_r\log|D u(x_0+r\theta)|
 = \frac{\lambda-1}{r}+O(r^{-1+\delta}).
$$
This last estimate is precisely what makes the singular parts of
$K/H$ and $(p-2)(N_*-1)/r$ cancel.

We emphasize the historical status of the consequences below. The planar theory of $p$-harmonic functions has a long independent history. Aronsson's work on singular $p$-harmonic functions and planar hodograph representations \cite{A1986,A1988,A1989}, together with the stream-function point of view developed by Aronsson--Lindqvist \cite{AL1988},
gave a detailed description of the local structure near critical points. The quasiregular-gradient approach of Bojarski--Iwaniec \cite{BI1987} and Manfredi \cite{M1988}, as well as the regularity work of Iwaniec--Manfredi \cite{IM1989} and Lewis's discussion of planar critical points \cite{L1994}, provide another standard route to the same two-dimensional structure; see also the monograph \cite{AIM2009} for the quasiregular background. Thus the
homogeneous blow-up statement in Corollary~\ref{cor:blowup} and the finite-order statement in Corollary~\ref{cor:finite-order} are consequences of Aronsson's representation theorem at critical points, together with the ordinary Taylor expansion at noncritical points. The planar unique continuation statement in Corollary~\ref{cor:ucp} is also
classical: It was proved by Alessandrini \cite{A1987} through the study of planar critical points, and it also follows from the quasiregular gradient method of Bojarski--Iwaniec \cite{BI1987} and Manfredi \cite{M1988}. The new point of this manuscript is therefore not a new proof of the planar structure theory, but the identification of the flux-normalized frequency and the exact gauge-corrected Almgren identity which packages the currently known structure in a frequency-function form.



More precisely, in the current manuscript, we employ earlier work, especially Aronsson's theorem, within a real-variable asymptotic framework, as developed in Section~\ref{sec:prelim}. The subsequent result is most appropriately interpreted as a restated  theorem on the flux-normalized growth, rather than as a genuinely new proof of the planar structure theory. The existence of the leading homogeneous profile is a consequence of Aronsson's original work. The primary purpose of this note is to isolate the notion of flux-normalized frequency and to clarify explicitly how Aronsson's prominent planar asymptotic theory entails a $p$-Almgren-type growth theorem in the planar setting.

\begin{thm}[Planar  $p$-Almgren theorem]\label{thm:main}
Let $u$ be a nonconstant $p$-harmonic function in a planar domain $\Omega$, and let $x_0\in\Omega$.  Put $v=u-u(x_0)$.  For all sufficiently small $r>0$, define
$$
  D(x_0,r)=\int_{B_r(x_0)} |D u|^p\,dx,
\quad 
  H(x_0,r)=\int_{\partial B_r(x_0)} |D u|^{p-2}v^2\,dS,
$$
and
$$
  N_*(x_0,r)=\frac{rD(x_0,r)}{H(x_0,r)}.
$$
If $u$ is not locally constant near $x_0$, then $H(x_0,r)>0$ for all sufficiently small $r$ and the limit
$$
  N_*(x_0,0+)=\lim_{r\to0}N_*(x_0,r)
$$
exists.

If $D u(x_0)\ne0$, then $ N_*(x_0,0+)=1.$
If $D u(x_0)=0$ and the Aronsson order of $x_0$ is $N\ge1$, then
$$
  N_*(x_0,0+)=\lambda_N(p),
$$
where
$$
  \lambda_N(p)=\frac{\beta_{N+1}(p)}{\beta_{N+1}(p)-1}
$$
and
$$
  \beta_{N+1}(p)
  =
  \frac{\sqrt{(p-2)^2+4(p-1)((N+1)/N)^2}-(p-2)}{2}.
$$
Moreover, in both cases there are constants $C>0$, $\delta>0$, and $r_0>0$ such that
$$
  |N_*(x_0,r)-N_*(x_0,0+)|\le Cr^\delta
  \qquad 0<r<r_0.
$$

In addition, the obstruction $\Theta$ defined as \eqref{eq:theta-defn} in Proposition~\ref{thm:exact-identity} belongs to $L^1(0,r_0)$, and the gauge-corrected frequency \eqref{eq:gauge-frequency} in Corollary \ref{cor:gauge-formal} is nondecreasing.
\end{thm}

When $p=2$, one has $A=|D u(x)|^{p-2}\equiv 1$, $K\equiv 0$, and 
$\Theta\equiv 0$. The identity therefore reduces exactly to the classical Almgren square formula. Thus the only genuinely nonlinear feature is the failure of the weighted height to differentiate without producing the logarithmic-gradient term $\partial_\nu(|Du|^{p-2})$.

For $0<a<r_0$, define the integrating factor
$$
  \Gamma_a(r)
  = \exp\left(\int_a^r \Theta(s)\,ds\right), \qquad a<r<r_0,
$$
and the corresponding corrected frequency
$$
        \mathcal N_a(r)=\Gamma_a(r)N_\ast(r).
$$
For $0<a<b<r_0$, define
$$
\mathcal R(a,b)
=
\int_a^b
\frac{s}{H(s)}
\int_{\partial B_s(x_0)}
A
\left(
\partial_\nu u-\frac{N_\ast(s)}{s}v
\right)^2
\,dS\,ds .
$$
More generally, for a fixed number $\lambda$, set
$$
\mathcal R_\lambda(a,b)
=
\int_a^b
\frac{s}{H(s)}
\int_{\partial B_s(x_0)}
A
\left(
\partial_\nu u-\frac{\lambda}{s}v
\right)^2
\,dS\,ds .
$$
Then as a consequence of Theorem~\ref{thm:main}, we obtain the following quantitative pinching of the flux-normalized frequency. 
\begin{thm} 
\label{thm:quantitative-pinching}
Let $u$ be a nonconstant planar $p$-harmonic function in $B_{r_0}(x_0)$.
Assume that $H(s)>0$ and that $\Theta\in L^1(a,b)$ for some
$0<a<b<r_0$. Then
$$
\mathcal N_a(b)-\mathcal N_a(a)= p
\int_a^b \Gamma_a(s) \frac{s}{H(s)} \int_{\partial B_s(x_0)} A \left( \partial_\nu u-\frac{N_\ast(s)}{s}v
\right)^2 \,dS\,ds .
$$
In particular,
$$
\mathcal N_a(b)-\mathcal N_a(a)\ge 0.
$$

Assume moreover that
$$
  \mathcal N_a(b)-\mathcal N_a(a)\le \epsilon
$$
and
$$
   \omega(a,b):=\int_a^b |\Theta(s)|\,ds\le \eta .
$$
Then
$$
  \mathcal R(a,b) \le  \frac{e^\eta}{p}\,\epsilon .
$$
If, in addition, $N_\ast(a)\le \Lambda$, and if we set
$ \lambda=N_\ast(a),$
then
$$
 \sup_{a\le s\le b}|N_\ast(s)-\lambda|  \le    e^\eta\epsilon+\Lambda(e^\eta-1),
$$
and
$$
\mathcal R_\lambda(a,b)
\le \frac{2e^\eta}{p}\,\epsilon
+ 2\log\frac ba
\left[ e^\eta\epsilon+\Lambda(e^\eta-1)
\right]^2 .
$$
\end{thm}

Let us also comment on the role of Theorem~\ref{thm:quantitative-pinching}.
Theorem~\ref{thm:main} gives the existence of the limiting flux-normalized
frequency and the corresponding gauge-corrected monotonicity. 
Theorem~\ref{thm:quantitative-pinching} is the quantitative form of this
monotonicity: If the corrected frequency has small increase on an annulus and
if the total gauge error
$$
        \int_a^b |\Theta(s)|\,ds
$$
is controlled, then the weighted radial homogeneity defect
$$
        \partial_\nu u-\frac{N_\ast(s)}{s}\bigl(u-u(x_0)\bigr)
$$
is small in an integral sense. Moreover, after freezing the exponent at
$\lambda=N_\ast(a)$, one obtains a corresponding estimate for
$$
        \partial_\nu u-\frac{\lambda}{s}\bigl(u-u(x_0)\bigr).
$$
Thus Theorem~\ref{thm:quantitative-pinching} may be viewed as an annular
almost-rigidity statement: Near constancy of the corrected frequency forces the
solution to be close, in the natural weighted radial sense, to a homogeneous
profile.

Finally, in Section~\ref{sec:higher} we explain which part of the argument is dimension-free and which part is genuinely planar. The algebraic flux-Almgren identity itself survives in every dimension. In fact, the same
weighted height, the same frequency $N_\ast$, and the same obstruction $\Theta$ lead formally to the same square identity in $\mathbb R^n$.
Therefore, the essential issue in dimensions $n\ge 3$ is not the Rellich--Pohozaev computation, but the integrability of $\Theta$ near critical points. In the plane this integrability follows from Aronsson's hodograph representation and the resulting logarithmic-gradient asymptotic. However, in higher dimensions such a structure theorem is not presently available in
general. Section~\ref{sec:higher} records the dimension-free identity and gives a conditional criterion, in terms of a nondegenerate homogeneous tangent and a weighted logarithmic-gradient asymptotic, under which the same conclusion holds. It seems to the author that, the higher-dimensional extension of the planar $p$-Almgren theorem is essentially a regularity and critical-set structure
problem.

\medskip

\noindent{\bf Acknowledgment}: The author would like to express heartfelt thanks to Nicola Fusco for introducing him to the field of calculus of variations and elliptic PDE, especially on the $p$-harmonic functions.

\section{Preliminaries} \label{sec:prelim}

Throughout the paper, $1<p<\infty$.  For $x_0\in\R^n$ and $r>0$, we set
$$
  B_r(x_0)=\{x\in\R^n: |x-x_0|<r\},
  \qquad   \partial B_r(x_0)=\{x\in\R^n: |x-x_0|=r\}.
$$
When $x_0=0$, simply write $B_r=B_r(0)$.  The outward unit normal on $\partial B_r(x_0)$ is denoted by $\nu$, and
$$
  \partial_\nu u=D u\cdot \nu.
$$
Surface measure on $\partial B_r(x_0)$ is denoted by $dS$.

\begin{defn}[Weak $p$-harmonicity]
Let $\Omega\subset\R^2$ be open.  A function $u\in W^{1,p}_{\mathrm{loc}}(\Omega)$ is called $p$-harmonic in $\Omega$ if
$$
  \int_\Omega |D u|^{p-2}D u\cdot D\varphi\,dx=0
$$
for every test function $\varphi\in C_c^\infty(\Omega)$.
\end{defn}

We shall use the following standard regularity and planar structure facts.  The local $C^{1,\alpha}$ regularity is due to the general theory of degenerate elliptic equations; see, for example, DiBenedetto \cite{D1983}, Lewis \cite{L1983}, Tolksdorf \cite{T1984}, as well as the nice exposition by Lindqvist \cite{L2006}.  The specifically planar assertions rely on Bojarski--Iwaniec \cite{BI1987}, Manfredi \cite{M1988}, and Aronsson \cite{A1989}. We also refer to the sharp regularity result of Iwaniec and Manfredi \cite{IM1989}. 

\begin{thm}[Regularity of planar $p$-harmonic functions] \label{thm:aronsson-package}
Let $u$ be a nonconstant $p$-harmonic function in a planar domain $\Omega$.
\begin{enumerate}[\textup{(\roman*)}]
\item The critical set
$$
  \crit(u):=\{x\in\Omega:D u(x)=0\}
$$
is discrete.  In particular, if $x_0\in\crit(u)$, then there exists $\rho>0$ such that $D u\ne0$ in $B_\rho(x_0)\setminus\{x_0\}$.

\item If $x_0\notin\crit(u)$, then $u$ is real analytic in a neighborhood of $x_0$.

\item Let $x_0\in\crit(u)$.  Then there are an integer $N\ge1$, a number $\delta>0$, and a nonzero homogeneous $p$-harmonic function $P$ in $\R^2$ such that, after subtracting the constant $u(x_0)$,
$$
  P(r\theta)=r^{\lambda_N(p)}\Phi(\theta),
  \qquad \theta\in\mathbb S^1,
$$
and, uniformly for $\theta\in\mathbb S^1$ as $r\to0$,
$$
  u(x_0+r\theta)-u(x_0)=P(r\theta)+O(r^{\lambda_N(p)+\delta}),
$$
$$
  D u(x_0+r\theta)=D P(r\theta)+O(r^{\lambda_N(p)-1+\delta}).
$$
Moreover $|D P(\theta)|>0$ on $\mathbb S^1$, and the hodograph representation gives the logarithmic-gradient expansion
$$
  \partial_r\log |D u(x_0+r\theta)|
  =
  \frac{\lambda_N(p)-1}{r}+O(r^{-1+\delta})
$$
uniformly in $\theta$.
\end{enumerate}
Here
$$
  \lambda_N(p)=\frac{\beta_{N+1}(p)}{\beta_{N+1}(p)-1},
$$
with
$$
  \beta_{N+1}(p)
  =
  \frac{\sqrt{(p-2)^2+4(p-1)((N+1)/N)^2}-(p-2)}{2}.
$$
For $p=2$, the same statement is understood in the classical harmonic sense and gives $\lambda_N(2)=N+1$.
\end{thm}

\begin{rem}
The integer $N$ in Theorem \ref{thm:aronsson-package} is Aronsson's order of the critical point.  Aronsson defines it through the first nonvanishing term in the quasiregular representation of the complex gradient  The formula for $\beta_{N+1}(p)$ comes from the constant-coefficient hodograph equations in Aronsson's representation theorem.  Then the leading mode is $N+1$ rather than $N$, which is why the index $N+1$ appears.
\end{rem}

\begin{rem}
The uniform logarithmic-gradient expansion is a consequence of the convergent hodograph series.  It is stronger than a mere $C^1$ asymptotic and is the point where  the planar proof uses Aronsson's representation theorem rather than only general regularity.
\end{rem}

\subsection{The flux-normalized quantities}

Let $u$ be $p$-harmonic in $B_R(x_0)\subset\R^2$.  We always center the height at the value $u(x_0)$.  Thus set
$$
  v(x)=u(x)-u(x_0).
$$
For $0<r<R$, define
$$
  D(r)=\int_{B_r(x_0)} |D u|^p\,dx
$$
and
$$
  H(r)=\int_{\partial B_r(x_0)} |D u|^{p-2}v^2\,dS.
$$
Whenever $H(r)>0$, define the flux-normalized frequency
$$
  N_*(r)=\frac{rD(r)}{H(r)}.
$$
We also write
$$
  A(x)=|D u(x)|^{p-2}.
$$
Thus
$$
  H(r)=\int_{\partial B_r(x_0)} A v^2\,dS.
$$

\begin{lem} \label{lem:height-positive}
Assume that $u$ is not identically equal to $u(x_0)$ in $B_R(x_0)$.  Then $H(r)>0$ for every sufficiently small admissible radius $r>0$.
\end{lem}

\begin{proof}
If $x_0\notin\crit(u)$, then $|D u|$ is bounded below near $x_0$, and the assertion follows since $v$ cannot vanish identically on all small circles unless $u\equiv u(x_0)$ locally.

Suppose $x_0\in\crit(u)$.  By Theorem \ref{thm:aronsson-package}, the critical point is isolated, so $A>0$ on $\partial B_r(x_0)$ for all sufficiently small $r>0$.  If $H(r)=0$, then $v=0$ on $\partial B_r(x_0)$.  Since $u$ minimizes the $p$-Dirichlet energy among functions with the same boundary values, the unique minimizer with zero boundary values is $v\equiv0$ in $B_r(x_0)$.  This contradicts the assumption that $u$ is not locally constant.  Hence $H(r)>0$.
\end{proof}

\subsection{Basic identities}

The first identity is the boundary flux identity.

\begin{lem} \label{lem:flux}
Let $u$ be $p$-harmonic in $B_R(x_0)$.  For every $0<r<R$ for which the trace identity is valid,
$$
  D(r)=\int_{\partial B_r(x_0)} A v\partial_\nu u\,dS.
$$
In particular the identity holds for all small $r$ in the setting of Theorem \ref{thm:main} below.
\end{lem}

\begin{proof}
If $u$ is smooth and $D u\ne0$ in a neighborhood of $\overline{B_r(x_0)}$, then
$$
  \operatorname{div}(AD u)=0
$$
pointwise, and the divergence theorem gives
$$
  0=\int_{B_r(x_0)} \operatorname{div}(vAD u)\,dx
    -\int_{B_r(x_0)} AD u\cdot D v\,dx.
$$
Since $D v=D u$, this is exactly the desired identity.

At a planar critical center, choose $r$ so small that $x_0$ is the only critical point in $B_r(x_0)$.  Apply the preceding smooth argument on the annulus
$$
  B_r(x_0)\setminus \overline{B_\eps(x_0)}.
$$
The inner boundary contribution is
$$
  \int_{\partial B_\eps(x_0)} A v\partial_\nu u\,dS.
$$
By the asymptotic expansion in Theorem \ref{thm:aronsson-package}, if the critical homogeneity is $\lambda>1$, then
$$
  |v|=O(\eps^\lambda),
  \qquad
  |D u|=O(\eps^{\lambda-1}).
$$
Thus the absolute value of the inner boundary contribution is bounded by
$$
  C\eps\,\eps^\lambda\,\eps^{(p-1)(\lambda-1)}
  =C\eps^{p\lambda-p+2},
$$
which tends to zero.  Letting $\eps\to0$ gives the identity.  The noncritical case follows from the smooth argument, and the general weak case follows by the standard regularization approximation used for the $p$-Laplace equation.
\end{proof}

The second identity is the planar Pohozaev identity.  We state it in the form needed here.

\begin{lem} \label{lem:pohozaev}
Let $u$ be $p$-harmonic in $B_R(x_0)$.  For every sufficiently small admissible radius $r$ in the planar setting considered below,
$$
  D'(r)=\frac{2-p}{r}D(r)+p\int_{\partial B_r(x_0)} A(\partial_\nu u)^2\,dS.
$$
\end{lem}

\begin{proof}
For smooth solutions with $D u\ne0$ this is the standard Rellich-Pohozaev identity obtained by testing the stationarity of the $p$-energy under the radial domain variation $x\mapsto x+t(x-x_0)$.  In the weak setting one obtains the same formula by regularizing the integrand, for instance replacing $|D u|^p$ by $(|D u|^2+\eps)^{p/2}$, applying the domain variation to the smooth regularized minimizers, and letting $\eps\to0$; see Hardt--Lin \cite[Lemma 4.1]{HL1987} and the exposition in Maldonado \cite[Section 4]{M2009}.

In the present planar application, if $x_0$ is not critical, $u$ is smooth near $x_0$ and the identity is classical.  If $x_0$ is a critical point, then by Theorem \ref{thm:aronsson-package} the only critical point in a small ball is $x_0$.  Apply the smooth identity on the annulus $B_r(x_0)\setminus \overline{B_\eps(x_0)}$ and let $\eps\to0$.  The inner boundary terms vanish since Theorem \ref{thm:aronsson-package} gives
$$
  |u-u(x_0)|=O(\eps^\lambda),
  \qquad
  |D u|=O(\eps^{\lambda-1})
$$
with $\lambda>1$ at a critical point.  Thus every inner boundary term is bounded by a positive power of $\eps$ and tends to zero.
\end{proof}

\section{The exact flux-Almgren identity}

Let
$$
  G(r)=\int_{\partial B_r(x_0)} A(\partial_\nu u)^2\,dS
$$
and
$$
  K(r)=\int_{\partial B_r(x_0)} v^2\partial_\nu A\,dS;
$$
recall that $  A(x)=|D u(x)|^{p-2}.$
The term $K(r)$ can be understood in the classical sense for all small $r$ around a planar critical point, since $D u\ne0$ on the circle and $u$ is real analytic away from the critical point.

\begin{lem}\label{lem:height-derivative}
For every small admissible radius,
$$
  H'(r)=\frac1r H(r)+2D(r)+K(r).
$$
\end{lem}

\begin{proof}
Write $x=x_0+r\theta$, $\theta\in\mathbb S^1$.  Then
$$
  H(r)=r\int_{\mathbb S^1} A(x_0+r\theta)v(x_0+r\theta)^2\,d\theta.
$$
Differentiating under the integral sign gives
$$
  H'(r)=\frac1rH(r)+\int_{\partial B_r(x_0)} \partial_\nu(A v^2)\,dS.
$$
Since
$$
  \partial_\nu(A v^2)=v^2\partial_\nu A+2Av\partial_\nu u,
$$
the flux identity gives
$$
  \int_{\partial B_r(x_0)}2Av\partial_\nu u\,dS=2D(r).
$$
This proves the formula.
\end{proof}

\begin{prop}\label{thm:exact-identity}
Let $u$ be a nonconstant planar $p$-harmonic function in $B_R(x_0)$, and suppose that $u$ is not locally constant at $x_0$.  For all sufficiently small admissible radii, define
\begin{equation}\label{eq:theta-defn}
      \Theta(r)  =  \frac{K(r)}{H(r)}  -  \frac{(p-2)(N_*(r)-1)}{r}.
\end{equation}
Then
$$
  N_*'(r)+N_*(r)\Theta(r)
  =
  \frac{pr}{H(r)}
  \int_{\partial B_r(x_0)}
  A\left(\partial_\nu u-\frac{N_*(r)}{r}v\right)^2\,dS.
$$
In particular,
$$
  N_*'(r)+N_*(r)\Theta(r)\ge0.
$$
\end{prop}

\begin{proof}
For readability write $N=N_*(r)$.  By definition,
$$
  \frac{N'}{N}=\frac1r+\frac{D'}{D}-\frac{H'}{H}.
$$
Using Lemmas \ref{lem:pohozaev} and \ref{lem:height-derivative},
$$
  \frac{N'}{N}
  =
  \frac{2-p}{r}+p\frac{G}{D}-2\frac{D}{H}-\frac{K}{H}.
$$
Since $D/H=N/r$, this becomes
$$
  N'
  =
  N\left(p\frac{G}{D}-\frac{2N+p-2}{r}-\frac{K}{H}\right).
$$
Therefore
$$
  N'+N\Theta
  =
  N\left(p\frac{G}{D}-\frac{pN}{r}\right).
$$
On the other hand, by the flux identity,
$$
  \int_{\partial B_r}A v\partial_\nu u\,dS=D.
$$
Hence
\begin{align*}
&\frac{pr}{H}\int_{\partial B_r} A\left(\partial_\nu u-\frac{N}{r}v\right)^2\,dS \\
&\quad =\frac{pr}{H}\left(G-2\frac{N}{r}D+\frac{N^2}{r^2}H\right) \\
&\quad =N\left(p\frac{G}{D}-\frac{pN}{r}\right).
\end{align*}
The two expressions are equal.
\end{proof}

\begin{cor} \label{cor:gauge-formal}
Assume that $\Theta\in L^1(0,r_0)$.  Then
\begin{equation}\label{eq:gauge-frequency}
r\mapsto  \exp\left(\int_{r_0}^r \Theta(s)\,ds\right)N_*(r) 
\end{equation}

is nondecreasing on $(0,r_0)$.
\end{cor}

\begin{proof}
Multiply the identity in Proposition~\ref{thm:exact-identity} by the integrating factor
$$
  \exp\left(\int_{r_0}^r\Theta(s)\,ds\right).
$$
The right-hand side remains nonnegative.
\end{proof}

\section{The planar $p$-Almgren theorem}

We now prove the main result.

\subsection{Proof of Theorem~\ref{thm:main} and its corollaries}

\begin{proof}[Proof of Theorem~\ref{thm:main}]
After translating the point and subtracting a constant, assume $x_0=0$ and $u(0)=0$.

First suppose $D u(0)\ne0$.  Since the equation is uniformly elliptic near $0$, the function is smooth there.  We have
$$
  u(x)=\ell(x)+O(|x|^2),
  \qquad
  D u(x)=D\ell(x)+O(|x|),
$$
where $\ell(x)=D u(0)\cdot x$ is a nonzero linear function.  Hence
$$
  D(r)=D_0r^2(1+O(r))
$$
and
$$
  H(r)=H_0r^3(1+O(r)).
$$
For the linear profile $\ell$, the flux identity gives $D_0=H_0$.  Therefore
$$
  N_*(r)=1+O(r).
$$
Also $\partial_\nu\log |D u|=O(1)$, whence $K(r)/H(r)=O(1)$ and $\Theta(r)=O(1)$.  Thus $\Theta\in L^1(0,r_0)$.

Now suppose $D u(0)=0$.  By Theorem \ref{thm:aronsson-package}, there is a nonzero homogeneous $p$-harmonic function
$$
  P(r\theta)=r^\lambda\Phi(\theta),
  \qquad \lambda=\lambda_N(p)>1,
$$
and a number $\delta>0$ such that
$$
  u(r\theta)=P(r\theta)+O(r^{\lambda+\delta})
$$
and
$$
  D u(r\theta)=D P(r\theta)+O(r^{\lambda-1+\delta})
$$
uniformly in $\theta\in\mathbb S^1$.  Since
$$
  |D P(r\theta)|=r^{\lambda-1}|D P(\theta)|
$$
and $|D P(\theta)|>0$ on $\mathbb S^1$, we get
$$
  |D u(r\theta)|^p
  =
  r^{p(\lambda-1)}|D P(\theta)|^p(1+O(r^\delta)).
$$
Thus
$$
  D(r)=D_0 r^{p\lambda-p+2}(1+O(r^\delta)),
$$
where
$$
  D_0=\frac1{p\lambda-p+2}\int_{\mathbb S^1}|D P(\theta)|^p\,d\theta.
$$
Similarly,
$$
  H(r)=H_0 r^{p\lambda-p+3}(1+O(r^\delta)),
$$
where
$$
  H_0=\int_{\mathbb S^1}|D P(\theta)|^{p-2}\Phi(\theta)^2\,d\theta.
$$
For the homogeneous profile $P$, one has
$ \partial_\nu P=\frac{\lambda}{r}P.$
The flux identity applied to $P$ gives
$$
  \int_{B_r}|D P|^p\,dx
  =
  \frac{\lambda}{r}\int_{\partial B_r}|D P|^{p-2}P^2\,dS.
$$
Comparing the powers of $r$ yields $D_0=\lambda H_0$.  Therefore
$$
  N_*(r)=\frac{rD(r)}{H(r)}=\lambda+O(r^\delta).
$$

It remains to check the obstruction.  By Theorem \ref{thm:aronsson-package},
$$
  \partial_r\log |D u(r\theta)|
  =
  \frac{\lambda-1}{r}+O(r^{-1+\delta})
$$
uniformly in $\theta$.  Since
$$
  \partial_\nu A=(p-2)A\partial_\nu\log |D u|,
$$
we obtain
$$
  \frac{K(r)}{H(r)}
  =
  \frac{(p-2)(\lambda-1)}{r}+O(r^{-1+\delta}).
$$
On the other hand,
$$
  \frac{(p-2)(N_*(r)-1)}{r}
  =
  \frac{(p-2)(\lambda-1)}{r}+O(r^{-1+\delta}).
$$
Subtracting gives
$$
  \Theta(r)=O(r^{-1+\delta}).
$$
Thus $\Theta\in L^1(0,r_0)$.  The gauge-corrected monotonicity follows from Corollary \ref{cor:gauge-formal}.
\end{proof}

Now we can reformulate the following known results in frequency language. 

\begin{cor}[Homogeneous blow-up]\label{cor:blowup}
Let $u$, $x_0$, and $\lambda=N_*(x_0,0+)$ be as in Theorem \ref{thm:main}.  If $u$ is not locally constant near $x_0$, then there exists a nonzero homogeneous $p$-harmonic function $P$ of degree $\lambda$ such that
$$
  \frac{u(x_0+r\,\cdot)-u(x_0)}{r^\lambda}
  \longrightarrow P
$$
locally uniformly and in $C^1$ on compact subsets of $\R^2\setminus\{0\}$ as $r\to0$.
\end{cor}

\begin{proof}
If $D u(x_0)\ne0$, this follows from the Taylor expansion with $P$ equal to the linear part.  If $D u(x_0)=0$, it is exactly the leading term in Aronsson's singular expansion quoted in Theorem \ref{thm:aronsson-package}.
\end{proof}

\begin{cor}[Finite order of vanishing]\label{cor:finite-order}
Let $u$ be a nonconstant planar $p$-harmonic function and $x_0\in\Omega$.  If $u(x_0)=0$, then $u$ has finite vanishing order at $x_0$ in the following sense: either $D u(x_0)\ne0$ and the order is $1$, or $x_0$ is critical and
$$
  |u(x_0+r\theta)|\le Cr^{\lambda_N(p)},
  \qquad
  \sup_{\theta\in\mathbb S^1}|u(x_0+r\theta)|\ge c r^{\lambda_N(p)}
$$
for all sufficiently small $r>0$, with constants $c,C>0$.
\end{cor}

\begin{proof}
The upper bound follows from the asymptotic expansion.  For the lower bound, the leading profile $P$ is nonzero, hence $\sup_{\mathbb S^1}|P|>0$.  The convergence in Corollary \ref{cor:blowup} gives the result for small $r$.
\end{proof}

\begin{cor}[Unique continuation in the plane]\label{cor:ucp}
Let $\Omega\subset\R^2$ be connected, and let $u$ be $p$-harmonic in $\Omega$.  If $u$ is constant on a nonempty open subset of $\Omega$, then $u$ is constant in all of $\Omega$.  In particular, if $u=0$ on a nonempty open subset, then $u\equiv0$ in $\Omega$.
\end{cor}

\begin{proof}
Assume, for contradiction, that $u$ is not constant in $\Omega$ but is constant, say $u=c$, on a nonempty open set $U\subset\Omega$.  Then $D u=0$ at every point of $U$.  This contradicts the discreteness of the critical set for a nonconstant planar $p$-harmonic function, stated in Theorem \ref{thm:aronsson-package}.  Therefore $u$ is constant in $\Omega$.
\end{proof}

\begin{rem}
This is the usual planar unique continuation principle.  The proof above uses the same planar structure theorem that underlies the flux-frequency theorem.  Equivalently, Corollary \ref{cor:finite-order} says that a nonconstant planar $p$-harmonic function cannot vanish to infinite order at a point at which it is not locally zero.
\end{rem}

\subsection{Quantitative pinching of the flux-normalized frequency}
\label{subsec:quantitative-pinching}

In this subsection we prove a quantitative consequence of the exact identity.
The point is that pinching of the corrected frequency is exactly equivalent to smallness of the weighted radial homogeneity defect.

\begin{proof}[Proof of Theorem~\ref{thm:quantitative-pinching}]
By Proposition~\ref{thm:exact-identity},
$$
  N_\ast'(s)+\Theta(s)N_\ast(s)   =  \frac{ps}{H(s)}  \int_{\partial B_s(x_0)}  A   \left(   \partial_\nu u-\frac{N_\ast(s)}{s}v
   \right)^2
   \,dS .
$$
Multiplying by $\Gamma_a(s)$, we obtain
$$
  \frac{d}{ds}\left(\Gamma_a(s)N_\ast(s)\right)  =   p\Gamma_a(s)   \frac{s}{H(s)}  \int_{\partial B_s(x_0)}  A  \left(    \partial_\nu u-\frac{N_\ast(s)}{s}v  \right)^2   \,dS .
$$
Integrating from $a$ to $b$ gives the identity.

Since
$$
   e^{-\eta}\le \Gamma_a(s)\le e^\eta   \qquad a\le s\le b,
$$
the pinching assumption gives
$$
   p e^{-\eta}\mathcal R(a,b)   \le   \mathcal N_a(b)-\mathcal N_a(a)  \le  \epsilon .
$$
Hence
$$
        \mathcal R(a,b)\le \frac{e^\eta}{p}\epsilon .
$$

It remains to prove the fixed-exponent estimate. Since $\mathcal N_a$ is
nondecreasing,
$$
  N_\ast(a)  \le   \mathcal N_a(s) \le    N_\ast(a)+\epsilon   \qquad a\le s\le b .
$$
Since $N_\ast(s)=\Gamma_a(s)^{-1}\mathcal N_a(s)$, we get
$$
\begin{aligned}
|N_\ast(s)-N_\ast(a)|
&\le \Gamma_a(s)^{-1} |\mathcal N_a(s)-N_\ast(a)|
+ N_\ast(a)|\Gamma_a(s)^{-1}-1|  \\
&\le e^\eta\epsilon+\Lambda(e^\eta-1).
\end{aligned}
$$
Now write
$$
\partial_\nu u-\frac{\lambda}{s}v
=
\left(
\partial_\nu u-\frac{N_\ast(s)}{s}v
\right)
+
\frac{N_\ast(s)-\lambda}{s}v .
$$
Using $(X+Y)^2\le 2X^2+2Y^2$, we obtain
$$
\begin{aligned}
\mathcal R_\lambda(a,b)
&\le
2\mathcal R(a,b)
+
2\int_a^b
\frac{s}{H(s)}
\int_{\partial B_s(x_0)}
A
\frac{(N_\ast(s)-\lambda)^2}{s^2}v^2
\,dS\,ds  \\
&=
2\mathcal R(a,b)
+
2\int_a^b
\frac{(N_\ast(s)-\lambda)^2}{s}\,ds .
\end{aligned}
$$
The previous two estimates give
$$
\mathcal R_\lambda(a,b)
\le
\frac{2e^\eta}{p}\epsilon
+
2\log\frac ba
\left[
e^\eta\epsilon+\Lambda(e^\eta-1)
\right]^2 .
$$
This proves the theorem.
\end{proof}

\begin{cor} 
\label{cor:dyadic-pinching}
Let $u$ be a nonconstant planar $p$-harmonic function near $x_0$, and assume
that $x_0\in \operatorname{Crit}(u)$. Let
$$
        \lambda=N_\ast(x_0,0+).
$$
Then there exist constants $C>0$, $\delta>0$, and $r_0>0$ such that, for
$0<r<r_0/2$,
$$
\int_r^{2r} \frac{s}{H(s)}\int_{\partial B_s(x_0)} A\left(\partial_\nu u-\frac{N_\ast(r)}{s}v
\right)^2\,dS\,ds\le C r^\delta .
$$
\end{cor}

\begin{proof}
By the proof of Theorem~\ref{thm:main}, after possibly
decreasing $r_0$,
$$
 N_\ast(s)=\lambda+O(s^\delta),  \qquad \Theta(s)=O(s^{-1+\delta})  \qquad 0<s<r_0 .
$$
Hence
$$
   \int_r^{2r}|\Theta(s)|\,ds\le C r^\delta
$$
and
$$
  \Gamma_r(2r)=1+O(r^\delta).
$$
Therefore
$$
\mathcal N_r(2r)-\mathcal N_r(r) = \Gamma_r(2r)N_\ast(2r)-N_\ast(r) = O(r^\delta).
$$
Applying Theorem~\ref{thm:quantitative-pinching} on the annulus
$[r,2r]$, with $a=r$, $b=2r$, and $\lambda=N_\ast(r)$, gives the claim.
\end{proof}

\section{Further remarks on higher-dimensional cases}\label{sec:higher}

The flux normalization indeed makes sense in every dimension.  Let $u$ be $p$-harmonic in $B_R(x_0)\subset\R^n$ and put $v=u-u(x_0)$,
$$
  D(r)=\int_{B_r(x_0)}|D u|^p\,dx,
  \qquad
  H(r)=\int_{\partial B_r(x_0)}|D u|^{p-2}v^2\,dS.
$$
The same computation gives, formally,
$$
  N_*'(r)+N_*(r)\Theta(r)
  =
  \frac{pr}{H(r)}\int_{\partial B_r(x_0)}|D u|^{p-2}
  \left(\partial_\nu u-\frac{N_*(r)}{r}v\right)^2\,dS,
$$
where
$$
  N_*(r)=\frac{rD(r)}{H(r)}
$$
and
$$
  \Theta(r)=
  \frac{\int_{\partial B_r}v^2\partial_\nu(|D u|^{p-2})\,dS}
  {\int_{\partial B_r}|D u|^{p-2}v^2\,dS}
  -
  \frac{(p-2)(N_*(r)-1)}{r}.
$$
The dimension only enters through the two identities
$$
  D'(r)=\frac{n-p}{r}D(r)+p\int_{\partial B_r}|D u|^{p-2}(\partial_\nu u)^2\,dS
$$
and
$$
  H'(r)=\frac{n-1}{r}H(r)+2D(r)+\int_{\partial B_r}v^2\partial_\nu(|D u|^{p-2})\,dS.
$$
Thus the obstruction is structurally identical in all dimensions.

There are two cases in which the higher-dimensional conclusion follows immediately. First, if $D u(x_0)\ne0$, then the equation is uniformly elliptic near $x_0$, and the same Taylor expansion as in the planar proof gives
$$
  N_*(r)=1+O(r),
  \qquad
  \Theta\in L^1(0,r_0).
$$

Second, if one already knows that $u$ has a homogeneous asymptotic expansion
$$
  u(x_0+r\theta)-u(x_0)=r^\lambda\Phi(\theta)+O(r^{\lambda+\delta})
$$
with the corresponding weighted logarithmic-gradient expansion
$$
  \frac{\int_{\partial B_r}|D u|^{p- 2}v^2\partial_\nu\log|D u|\,dS}  {\int_{\partial B_r}|D u|^{p-2}v^2\,dS}
  =  \frac{\lambda-1}{r}+O(r^{-1+\delta}),
$$
then the proof of Theorem \ref{thm:main} gives
$$
  N_*(r)=\lambda+O(r^\delta),
  \qquad  \Theta\in L^1(0,r_0).
$$

The difficulty is that, for $n\ge3$, no analogue of Aronsson's planar hodograph representation is known.  The critical set need not be discrete, and the weight $|D u|^{p-2}$ degenerates or becomes singular on the critical set.  Lindqvist emphasizes that several properties known in the plane, including unique continuation, are not known in space \cite[Section 7]{L2006}.  Granlund--Marola also formulate the frequency approach to unique continuation for the $p$-Laplace equation as an open problem beyond the available planar methods \cite{GM2014}.  Therefore the full higher-dimensional critical-point version of Theorem \ref{thm:main}  presently seems to be a  regularity problem, i.e. 
$$
 \text{whether we have} \  \Theta\in L^1(0,r_0)
  \ \text{ or not.}
$$
Let us explain this in more detail below.
 
\subsection{The  flux-Almgren identity in $\mathbb R^n$}
\label{subsec:dimension-free-identity}

We now record the exact identity in arbitrary dimension. This is useful since it
isolates the only obstruction to a genuine Almgren monotonicity formula.

Let $n\ge 3$, let $u$ be $p$-harmonic in $B_R(x_0)\subset \mathbb R^n$, and set
$$
 v=u-u(x_0),   \qquad  A=|Du|^{p-2}.
$$

Define $D(r), H(r), N_\ast(r), G(r), K(r)  $ similarly as before, and still set
$$
 \Theta(r) = \frac{K(r)}{H(r)} - \frac{(p-2)(N_\ast(r)-1)}{r}.
$$
Then one has the following version of Proposition~\ref{thm:exact-identity} in the higher dimension with the same proof. 
\begin{prop} \label{prop:dimension-free-identity}
Assume that $u$ is smooth and $Du\neq 0$ in an annulus containing
$\partial B_r(x_0)$. Then
$$
  D'(r)=\frac{n-p}{r}D(r)+pG(r),
$$
and
$$
   H'(r) = \frac{n-1}{r}H(r)+2D(r)+K(r).
$$
Consequently, at every such radius for which $H(r)>0$,
$$
  N_\ast'(r)+N_\ast(r)\Theta(r)  =  \frac{pr}{H(r)}
 \int_{\partial B_r(x_0)}  A \left(  \partial_\nu u-\frac{N_\ast(r)}{r}v \right)^2   \,dS .
$$
In particular,
$$
  N_\ast'(r)+N_\ast(r)\Theta(r)\ge 0.
$$
\end{prop}

\subsection{A higher-dimensional criterion for \texorpdfstring{$\Theta\in L^1$}{Theta n in L1}}
\label{subsec:theta-n-L1}

Proposition~\ref{prop:dimension-free-identity}  shows that the only obstruction to a monotonicity formula is
the integrability of $\Theta$. We now give a verifiable sufficient condition.

\begin{defn}
\label{def:nondegenerate-homogeneous-tangent}
Let $u$ be $p$-harmonic in $B_R(x_0)\subset \mathbb R^n$, and set
$ v=u-u(x_0).$
Assume $Du(x_0)=0$. We say that $u$ has a nondegenerate homogeneous tangent
of degree $\lambda>1$ at $x_0$, with rate $\delta>0$, if there exists a nonzero
homogeneous $p$-harmonic function $P$ of degree $\lambda$, namely
$$
 P(r\theta)=r^\lambda P(\theta),
 \qquad r>0,\quad \theta\in S^{n-1},
$$
such that $\inf_{\theta\in S^{n-1}} |DP(\theta)|>0,$
and, for some constant $C_0>0$,
$$
\sum_{j=0}^2 r^j   \left|D^j v(x_0+r\theta)-D^jP(r\theta)\right| \le  C_0 r^{\lambda+\delta}
$$
for every $\theta\in S^{n-1}$ and every sufficiently small $r>0$.
Here $D^0 f=f$.
\end{defn}

\begin{thm}
\label{thm:theta-n-L1-nondegenerate}
Let $u$ be $p$-harmonic in $B_R(x_0)\subset \mathbb R^n$, $n\ge 3$, and assume
that $Du(x_0)=0$. Suppose that $u$ has a nondegenerate homogeneous tangent
$P$ of degree $\lambda>1$ at $x_0$, with rate $\delta>0$, in the sense of
Definition~\ref{def:nondegenerate-homogeneous-tangent}. Then, for all sufficiently
small $r>0$, $H(r)>0$, $N_\ast(r)$ is well-defined, and
$$
N_\ast(r)=\lambda+O(r^\delta).
$$
Moreover, $  \Theta(r)=O(r^{-1+\delta}). $
Consequently,
$  \Theta\in L^1(0,r_0) $
for some $r_0>0$, and the gauge-corrected frequency
$$  r\mapsto  \exp\left(\int_{r_0}^r \Theta(s)\,ds\right)N_\ast(r)
$$
is nondecreasing on $(0,r_0)$.
\end{thm}

\begin{proof}
After translating and subtracting a constant, assume $x_0=0$ and $u(0)=0$.
Write $\alpha=p\lambda-p+n.$
Since $P$ is homogeneous of degree $\lambda$,
$$
   DP(r\theta)=r^{\lambda-1}DP(\theta).
$$
The nondegeneracy assumption gives
$$
 |DP(\theta)|\ge c_0>0  \qquad \theta\in S^{n-1}.
$$
Hence, for all sufficiently small $r>0$,
$$
  c r^{\lambda-1}    \le   |Du(r\theta)|  \le    C r^{\lambda-1}   \qquad \theta\in S^{n-1}.
$$
In particular $Du\neq 0$ in the punctured ball $B_{r_0}\setminus\{0\}$, after
possibly decreasing $r_0$. Thus $u$ is smooth there and the quantities $K(r)$
and $\Theta(r)$ are classically defined for $0<r<r_0$.

We first compute $D(r)$ and $H(r)$. By the expansion in
Definition~\ref{def:nondegenerate-homogeneous-tangent},
$$
 Du(r\theta)=DP(r\theta)+O(r^{\lambda-1+\delta}) =  r^{\lambda-1}DP(\theta)+O(r^{\lambda-1+\delta}).
$$
Since $|DP(\theta)|\ge c_0$, this implies
$$
  |Du(r\theta)|^p   =  r^{p(\lambda-1)}  |DP(\theta)|^p   \left(1+O(r^\delta)\right).
$$
Therefore, using polar coordinates,
\begin{align}
D(r)
&=\int_0^r\int_{S^{n-1}}|Du(t\theta)|^p t^{n-}\,d\theta\,dt                        \notag \\
&=\int_0^r t^{p(\lambda-1)+n-1}\left[\int_{S^{n-1}} |DP(\theta)|^p\,d\theta+O(t^\delta)
\right]dt \notag  \\
&=D_0 r^\alpha(1+O(r^\delta)), \label{eq:D-n}
\end{align}
where
$$
  D_0=\int_{B_1}|DP|^p\,dx>0.
$$

Similarly,
$$
  v(r\theta)=P(r\theta)+O(r^{\lambda+\delta})  =  r^\lambda P(\theta)+O(r^{\lambda+\delta}),
$$
and so
\begin{align}
H(r)
&=\int_{\partial B_r}
|Du|^{p-2}v^2\,dS     \notag         \\
&= r^{(p-2)(\lambda-1)+2\lambda+n-1}
\int_{S^{n-1}}|DP(\theta)|^{p-2}P(\theta)^2\,d\theta
\left(1+O(r^\delta)\right)         \notag    \\
&=H_0 r^{\alpha+1}(1+O(r^\delta)), \label{eq:H-n}
\end{align}
where
$$
H_0=  \int_{S^{n-1}}  |DP(\theta)|^{p-2}P(\theta)^2\,d\theta>0.$$
Since $P$ is homogeneous and $p$-harmonic,
$\partial_\nu P=\frac{\lambda}{r}P.$ Applying the flux identity to $P$ on $B_1$, we get
$$
 D_0 = \int_{\partial B_1}|DP|^{p-2}P\partial_\nu P\,dS
 = \lambda \int_{\partial B_1}|DP|^{p-2}P^2\,dS
= \lambda H_0.
$$
Consequently, \eqref{eq:D-n} and \eqref{eq:H-n} yileds
$$
 N_\ast(r) = \frac{rD(r)}{H(r)}  = c{D_0}{H_0}(1+O(r^\delta))   =   \lambda+O(r^\delta).
$$

It remains to estimate $K(r)/H(r)$. Since $A=|Du|^{p-2},$
we have
$$
\partial_\nu A =  (p-2)A\,\partial_\nu\log |Du|.
$$
We claim that
\begin{equation}\label{eq:claim-log}
  \partial_\nu\log |Du(r\theta)|  =  \frac{\lambda-1}{r}
 +   O(r^{-1+\delta})
\end{equation}
uniformly in $\theta\in S^{n-1}$. Indeed,
$$
\partial_\nu\log |Du| = \frac{Du\cdot D^2u\,\nu}{|Du|^2}.
$$
The $j=1,2$ estimates in Definition~\ref{def:nondegenerate-homogeneous-tangent}
give
$$
 Du(r\theta)=DP(r\theta)+O(r^{\lambda-1+\delta}),
$$
and
$$
D^2u(r\theta)=D^2P(r\theta)+O(r^{\lambda-2+\delta}).
$$
Since $|DP(r\theta)|\ge c_0 r^{\lambda-1}$, comparison with $P$ yields
$$
\partial_\nu\log |Du(r\theta)|
 =\partial_\nu\log |DP(r\theta)|
  + O(r^{-1+\delta}).
$$
Moreover, 
$$
|DP(r\theta)|=r^{\lambda-1}|DP(\theta)|,
$$
and therefore
$$
\partial_\nu\log |DP(r\theta)|=\frac{\lambda-1}{r}.
$$
This proves \eqref{eq:claim-log}.

Using \eqref{eq:claim-log},
\begin{align*}
\frac{K(r)}{H(r)}
&=\frac{\int_{\partial B_r}v^2\partial_\nu A\,dS}{\int_{\partial B_r}Av^2\,dS}                                                        \\
&=(p-2)\frac{\int_{\partial B_r}Av^2\partial_\nu\log |Du|\,dS}{\int_{\partial B_r}Av^2\,dS}       \\
&= (p-2)\left[\frac{\lambda-1}{r}+O(r^{-1+\delta})\right].
\end{align*}
On the other hand,
$$
  \frac{(p-2)(N_\ast(r)-1)}{r}  =   (p-2)  \left[   \frac{\lambda-1}{r}  +  O(r^{-1+\delta})
  \right].
$$
Subtracting gives
$$
  \Theta(r) = \frac{K(r)}{H(r)} -   \frac{(p-2)(N_\ast(r)-1)}{r}  =  O(r^{-1+\delta}).
$$
Since $\delta>0$, this belongs to $L^1(0,r_0)$. The monotonicity of the
gauge-corrected frequency follows from Proposition~\ref{prop:dimension-free-identity}
by multiplying the identity by the integrating factor
$$
 \exp\left(\int_{r_0}^r \Theta(s)\,ds\right).
$$
\end{proof}

\begin{rem}[A weaker analytic version]
\label{rem:weaker-theta-criterion}
The $C^2$-rate in Definition~\ref{def:nondegenerate-homogeneous-tangent} is used
only to obtain
$$ \partial_\nu\log |Du(r\theta)|   =  \frac{\lambda-1}{r}
   +    O(r^{-1+\delta}).
$$
Therefore Theorem~\ref{thm:theta-n-L1-nondegenerate} remains valid if the
$C^2$-asymptotic assumption is replaced by the two assumptions
$$
 u(r\theta)=P(r\theta)+O(r^{\lambda+\delta}), \qquad
Du(r\theta)=DP(r\theta)+O(r^{\lambda-1+\delta}),
$$
together with the weighted logarithmic-gradient estimate
$$
\frac{\int_{\partial B_r}|Du|^{p-2}v^2\partial_\nu\log |Du|\,dS}{\int_{\partial B_r}|Du|^{p-2}v^2\,dS}
=\frac{\lambda-1}{r}+O(r^{-1+\delta}).
$$
This seems to be the minimal condition needed for $\Theta\in L^1$.
\end{rem}

\end{document}